\documentclass[12pt]{amsart}

\usepackage{amsmath, amsfonts, amssymb, amsxtra, amsaddr}
\usepackage{mathrsfs}
\usepackage{hyperref}
\usepackage{bbm}
\usepackage{tikz}

\usepackage{fullpage}

\newcommand*\circled[1]{\tikz[baseline=(char.base)]{\node[shape=circle,draw,inner sep=1pt] (char) {#1};}}
\newcommand*\squared[1]{\tikz[baseline=(char.base)]{\node[shape=rectangle,draw,inner sep=2pt] (char) {#1};}}

\numberwithin{equation}{section}

\begin{document}

\newtheorem{definition}{Definition}[section]
\newtheorem{theorem}[definition]{Theorem}
\newtheorem{lemma}[definition]{Lemma}
\newtheorem{corollary}[definition]{Corollary}
\newtheorem{remark}[definition]{Remark}
\newtheorem{proposition}[definition]{Proposition}
\newtheorem{conjecture}[definition]{Conjecture}

\title{Inventory Accumulation with $k$ Products} 

\author{\small Cheng Mao and Tianyou Zhou}

\address{Massachusetts Institute of Technology}

\date{\today}

\begin{abstract}
Sheffield (2011) proposed an inventory accumulation model with two types of products to study the critical Fortuin-Kasteleyn model on a random planar map, and showed that a two-dimensional inventory accumulation trajectory in the discrete model scales to a correlated planar Brownian motion. In this work, we generalize the inventory model to $k$ types of products for any integer $k \ge 2$, and prove that the corresponding trajectory scales to a $k$-dimensional Brownian motion with an appropriate covariance matrix. 
\end{abstract}



\maketitle

\section{Introduction} \label{intro}

\emph{Planar maps} are connected planar graphs embedded into the two-dimensional sphere defined up to homeomorphisms of the sphere \cite{Tut63}. Extensive work has been done to study the scaling limits of random planar maps; see e.g. \cite{LeG10, LeGMen10, LeGMie11, LeGMie12, Mie13, BetJacMie14} for recent developments.
%e.g. see \cite{lgm12} for one successful approach. 
In the paper \cite{She11}, Sheffield introduced a new approach to study planar maps by constructing a bijection between the \emph{critical Fortuin-Kasteleyn (FK) cluster model} on a random planar map \cite{ForKas72, Gri06} and an inventory accumulation model at a last-in-first-out retailer with two types of products, called \emph{hamburgers} and \emph{cheeseburgers}. In this inventory model, production of a hamburger, production of a cheeseburger, consumption of a hamburger, consumption of a cheeseburger and consumption of the \emph{freshest} burger happen with respective probabilities $\frac 14, \frac 14, \frac{1-p}4, \frac{1-p}4$ and $\frac p2$ at each discrete time point, where the freshest burger is the most recently produced burger regardless of type. The paper proved that the evolution of the two-burger inventory in the infinite-volume version of the model scales to a two-dimensional Brownian motion with covariance depending on $p$. An interesting phase transition happens at $p = 1/2$. In particular, when $p \geq 1/2$, the burger inventory remains balanced as the time goes to infinity, i.e., the discrepancy between the two burgers remains small.

More recently, efforts have been made to understand variations of the inventory accumulation model and their connections to FK models on random planar maps. A conditional version of the hamburger-cheeseburger model was studied independently in \cite{GwyMaoSun15} and \cite{BerLasRay15}. In particular, \cite{GwyMaoSun15} showed that the inventory trajectory converges to a correlated Brownian motion conditioned to stay in a quadrant. In addition, the finite-volume version of the two-dimensional model was studied in \cite{GwySun15a,GwySun15b}, where the trajectory was shown to scale to a Brownian motion conditioned to return to the origin. Furthermore, the hamburger-cheeseburger model has also been applied to study other statistical physics models, e.g. the abelian sandpile model and the uniform spanning unicycle on random planar maps \cite{SunWil15}.

\subsection*{Our contributions}

First of all, apart from random planar maps, the inventory accumulation model is of its own interest in view of its nice properties and the scaling limit results. In this work, we explore a new direction to generalize the model, namely to study inventory accumulation with $k$ types of products (referred as burgers) for any integer $k \ge 2$. In particular, we will prove that the corresponding $k$-dimensional trajectory scales to a $k$-dimensional Brownian motion and identify its covariance matrix. A phase transition occurs at the critical probability $p = 1-1/k$ which generalizes the two-dimensional result. 

The high-level strategy of our proof for the k-dimensional scaling limit result is the same as that in dimension two. However, it is worth noting that many adjustments need to be made and we highlight some originality here. First, the calculation of the covariance matrix of the limiting Brownian motion is more complicated, as the interactions between different types of products are more complex in higher dimensions. Additionally, several new arguments are introduced, e.g. the inductive argument in the proof of Lemma~\ref{lemma3.5}, to generalize the proofs beyond dimension two. Furthermore, the monotonicity properties used in \cite[Section~3.4]{She11} do not hold in higher dimensions since there is no natural ordering of inventory stacks when $k > 2$. Instead, we make use of the property that if two inventory stacks are close in an appropriate sense, then they still stay close after adding the same product or order to each of them.

So far, little research has been devoted to higher-dimensional analogues of random planar maps, partly due to the difficulty of enumeration and lack of bijective representations. See \cite{BenCur11} for an interesting higher-dimensional result among the few. 
%Although it is not clear to us whether the $k$-product inventory model described here relates to any nice higher-dimensional generalization of random planar maps, we hope our work can be a small step towards higher dimensions.
We hope that our generalized model and results can be used to construct potentially interesting higher-dimensional objects, possibly as follows.

When there are only two types of products, i.e. hamburgers and cheeseburgers, the number of hamburgers and the number of cheeseburgers after $n$ steps can be interpreted as
two walks on $\mathbb Z$. As explained in \cite{She11}, each of these walks
separately encodes a tree (via a standard bijection between walks and
trees) along with a path tracing the boundary of the tree. Furthermore, one can form a larger graph, which contains a planar map as a subgraph, by starting with
these two trees and then adding an edge between a vertex on the first
tree and a vertex on the second tree if those vertices are both
visited at the same time by the traversing paths. This construction gives a bijection between burger-order sequences and planar maps; see \cite[Section~4]{She11} for more details.

It is straightforward to generalize this construction to our setting
to obtain $k$ trees along with extra edges joining vertices on
different trees. We are not aware of any natural physical
interpretation of the random graph obtained this way when $k > 2$, but
we feel that this might be an interesting avenue for future research.

\medskip

The rest of the paper is organized as follows.
In Section~\ref{setup}, we will describe the model in detail and state the main scaling limit theorem. Section~\ref{covariance} is devoted to computing the covariance matrix of the limiting Brownian motion. We prove various technical estimates in Section~\ref{excursion} and \ref{tail_estimates}, and finish the proof of the main theorem in Section~\ref{main_proof}.

\subsection*{Acknowledgement}
We thank Scott Sheffield and Xin Sun for suggesting and initiating this project. This paper could not be completed without their helpful ideas and suggestions. We thank the 2014 UROP+ program at Massachusetts Institute of Technology during which part of this work was completed, and thank Pavel Etingof and Slava Gerovitch for directing the program.

\section{Model setup and the main theorem} \label{setup}

We consider a last-in-first-out retailer with $k$ types of products, to which we refer as \emph{burger 1}, \dots, \emph{burger k}. 
To adapt the two-dimensional model introduced in \cite{She11}, we define an alphabet of symbols \[ \Theta = \big\{\circled{1}, \circled{2}, \dots, \circled{k}, \squared{1}, \squared{2}, \dots, \squared{k}, \squared{F}\big\} \]
which represent the $k$ types of burgers, the corresponding $k$ types of orders each of which consumes the most recently produced burger of the same type, and the ``flexible'' order which consumes the most recently produced burger regardless of type in the remaining burger stack.

A word in the alphabet $\Theta$ is a concatenation of symbols in $\Theta$ that represents a series of events happened at the retailer. For example, if $W = \circled{2} \, \circled{3} \, \squared{3} \, \circled{1} \, \squared{2} \, \squared{F}$, then the word $W$ represents the series of events: a burger 2 is produced, a burger 3 is produced, a burger 3 is ordered, a burger 1 is produced, a burger 2 is ordered and the freshest burger is ordered, which is burger 1 in this case.

To describe the evolution of burger inventory mathematically, we consider the collection $\mathcal{G}$ of (reduced) words in the alphabet $\Theta$ modulo the following relations
\begin{equation} \label{word_relation}
\circled{i} \, \squared{i} = \circled{i} \, \squared{F} = \varnothing \quad \text{ and } \quad
\circled{i} \, \squared{j} = \squared{j} \, \circled{i}
\end{equation}
where $1 \leq i,j \leq k$ and $i \neq j$. Intuitively, the first relation means that an order \squared{i} or \squared{F} consumes a preceding burger \circled{i}, and the second means that we move an order one position to the left if it does not consume the immediately preceding burger. For example,
\[ W = \circled{2} \, \circled{3} \, \squared{3} \, \circled{1} \, \squared{2} \, \squared{F} = \circled{2} \, \circled{1} \, \squared{2} \, \squared{F} = \circled{2} \, \squared{2} \, \circled{1} \,  \squared{F} = \varnothing, \]
where $\squared{3}$ consumes $\circled{3}$,  $\squared{2}$ consumes $\circled{2}$ and  $\squared{F}$ consumes $\circled{1}$. By the same argument as in the proof of \cite[Proposition~2.1]{She11}, we see that $\mathcal{G}$ is a semigroup with $\varnothing$ as the identity and concatenation as the binary operation.

Let $X(n)$ be i.i.d. random variables indexed by $\mathbb{Z}$ (i.e. time), each of which takes its value in $\Theta$ with respective probabilities
\[ \Big\{ \frac 1{2k}, \frac 1{2k}, \dots, \frac 1{2k}, \frac{1-p}{2k}, \frac{1-p}{2k}, \dots, \frac{1-p}{2k}, \frac p2 \Big\}. \]
Let $\mu$ denote the corresponding probability measure on the space $\Omega$ of maps from $\mathbb{Z}$ to $\Theta$. In this paper, we follow the convention that probabilities and expectations are with respect to $\mu$ unless otherwise mentioned. For $m \leq n$, we write
\[ X(m,n) := \overline{X(m) X(m+1) \cdots X(n)} \]
where $\overline{\ \cdot\ }$ means that a word is reduced modulo the relations \eqref{word_relation}. Then $X(m,n)$ describes the remaining orders and burgers (after all consumptions) between time $m$ and time $n$ at the retailer.

If a burger is added at time $m$ and consumed at time $n$, we define $\phi(m) = n$ and $\phi(n) = m$. Otherwise, if a burger at $m$ has no corresponding order, then $\phi(m) = \infty$, or if an order at $n$ has no corresponding burger, then $\phi(n) = - \infty$. Proposition 2.2 in \cite{She11} remains valid in this $k$-burger setting:

\begin{proposition} \label{proposition1.1}
It is $\mu$-almost sure that for every $m \in \mathbb{Z}$, $\phi(m)$ is finite. %In other words, $\phi$ is an involution on $\mathbb{Z}$ almost surely.
\end{proposition}

Since a slight modification of the original proof will work here, we only describe the ideas. Let $E_i$ be the event that every burger of type $i$ is ultimately consumed. It can be shown that the union of $E_i$'s has probability one, and since $E_i$'s are translation-invariant, the zero-one law implies that each of them occurs almost surely. A similar argument works for orders, so each $X(m)$ has a correspondence, which is the statement of Proposition~\ref{proposition1.1}.

Hence we can define
\[
Y(n):=
\begin{cases}
X(n) &  X(n) \neq \squared{F}, \\
\squared{i} &  X(n) = \squared{F}, X(\phi(n)) = \circled{i}.
\end{cases}
\]
Moreover, we define the \emph{semi-infinite burger stack} $X(-\infty,n)$ to be the sequence of $X(m)$ where $m \leq n$ and $\phi(m) > n$. It contains no orders almost surely since each order consumes an earlier burger at a finite time due to Proposition~\ref{proposition1.1}.  It is almost surely infinite, because otherwise the number of burgers minus the number of orders in $X(-\infty,n)$ is a simple random walk in $n$ and will visit $-1$ at a finite time almost surely, but an order added at or before that time will consume no burger which contradicts Proposition~\ref{proposition1.1}.

Next, we give definitions of several important discrete processes that will be shown to scale to Brownian motions.

\begin{definition} \label{def0}
For a word $W$ in the alphabet $\Theta$, we define $\mathcal{C}^i(W)$ to be the net burger \emph{count} of type $i$, i.e., the number of \emph{\circled{i}} minus the number of \emph{\squared{i}}. Also, we define $\mathcal{C}(W)$ to be the total burger count, i.e.,
\[ \mathcal{C}(W) := \sum_{i=1}^k \mathcal{C}^i(W). \]

If $W$ has no \emph{$\squared{F}$}, then for $1 \leq i \neq j \leq k$, we define $\mathcal{D}^{ij}(W)$ to be the net \emph{discrepancy} of burger $i$ over burger $j$, i.e., 
\[ \mathcal{D}^{ij}(W) := \mathcal{C}^i(W) - \mathcal{C}^j(W). \]
\end{definition}

\begin{definition} \label{def1}
Given the infinite $X(n)$ sequence, let $\mathcal{C}^i_n$ be the integer-valued process defined by $\mathcal{C}^i_0 = 0$ and $\mathcal{C}^i_n - \mathcal{C}^i_{n-1} = \mathcal{C}^i(Y(n))$ for all $n$. Let $\mathcal{C}_n := \sum_{i=1}^k \mathcal{C}^i_n$ and $\mathcal{D}^{ij}_n := \mathcal{C}^i_n - \mathcal{C}^j_n.$

For any integer $n$, we define two vector-valued processes $A_n$ and $\widetilde{A}_n$ by
\[ A_n := (\mathcal{D}^{12}_n, \mathcal{D}^{23}_n, \dots, \mathcal{D}^{k-1,k}_n, \mathcal{C}_n) \quad \text{ and } \quad \widetilde{A}_n := (\mathcal{C}^1_n, \mathcal{C}^2_n, \dots, \mathcal{C}^k_n). \]
We extend these definitions to real numbers by piecewise linear interpolation so that $t \mapsto A_t$ and $t \mapsto \widetilde{A}_t$ are infinite continuous paths.
\end{definition}

When $n>0$, we have $\mathcal{C}^i_n = \mathcal{C}^i(Y(1,n))$; when $n<0$, we have $\mathcal{C}^i_n = \mathcal{C}^i(Y(n+1,0))$; similarly for $\mathcal{C}_n$ and $\mathcal{D}^{ij}_n$. 
As shorthand notation, we write
\[ \mathcal{C}^i(m) = \mathcal{C}^i(Y(m)) \quad \text{ and } \quad \mathcal{C}^i(m,n) = \mathcal{C}^i(Y(m,n))\]
for $m \leq n$, and we let $\mathcal{C}(m), \mathcal{D}^{ij}(m), \mathcal{C}(m,n)$ and  $\mathcal{D}^{ij}(m,n)$ be defined similarly.

Note that the two processes $A_n$ and $\widetilde{A}_n$ actually code the same information about the evolution of the sequence $Y(n)$. Specifically, if we view $A_n$ and $\widetilde{A}_n$ as column vectors, then it follows from Definition \ref{def1} that $A_n = M \widetilde{A}_n$ where $M$ is a $k \times k$ invertible matrix defined by
\[
M_{ij} = 
\begin{cases} 
1 & i=j, 1\le i \le k-1, \\
-1 & i+1 = j, 1 \le i \le k-1, \\
1 & i=k, \\
0 & \text{otherwise}.
\end{cases}
\]
%\[
%M :=
%\begin{bmatrix}
%1 & -1 & 0 & 0 & \cdots & 0 \\
%0 & 1 & -1 & 0 & \cdots & 0 \\
%\vdots & \ddots & \ddots & \ddots & \ddots & \vdots \\
%0 & \cdots & 0 & 1 & -1 & 0 \\
%0 & \cdots & 0 & 0 & 1 & -1 \\
%1 & 1 & 1 & \cdots & 1 & 1
%\end{bmatrix}
%.
%\]
%The inverse of $M$ is given by
%\[
%M^{-1} = \frac 1k
%\begin{bmatrix}
%k-1 & k-2 & k-3 & \cdots & 2 & 1 & 1 \\
%-1 & k-2 & k-3 & \cdots & 2 & 1 & 1 \\
%-1 & -2 & k-3 & \cdots & 2 & 1 & 1 \\
%\vdots & \vdots & \vdots & \vdots & \vdots & \vdots & \vdots \\
%-1 & -2 & -3 & \cdots & -k+2 & 1 & 1 \\
%-1 & -2 & -3 & \cdots & -k+2 & -k+1 & 1
%\end{bmatrix}
%.
%\]

It is more natural to describe the evolution of $Y(1,n)$ by $\widetilde{A}_n$ as its $i$-th coordinate corresponds to the burger count of type $i$. However, $A_n$ gives another interesting perspective to view the stack $Y(1,n)$. Consider the line $\mathscr{L}$ through $(0,\dots,0)$ and $(1,\dots,1)$ in $\mathbb{R}^k$. Since $\mathcal{C}_n$ is a simple random walk along $\mathscr{L}$ and is independent of the other $k-1$ coordinates of $A_n$, we may view $A_n$ as the Cartesian product of a one-dimensional simple random walk and an independent walk on the perpendicular $(k-1)$-dimensional hyperplane. The idea of separating the net burger count from the net burger discrepancies is inherited from the two-dimensional case.

With the linear relation established between $A_n$ and $\widetilde{A}_n$, we are ready to state two equivalent versions of the main scaling limit theorem.

\begin{theorem}[Main theorem, version 1]
\label{main_theorem}
As $\varepsilon \to 0$, the random variables $\varepsilon A_{t/\varepsilon^2}$ converge in law (with respect to the $L^\infty$ metric on compact intervals) to 
\[(\mathbf{B}^1_{\alpha t}, B^2_t),\]
where $\mathbf{B}^1_t = (W_t^1, \dots, W_t^{k-1})$ is a $(k-1)$-dimensional Brownian motion with covariance
\[
\operatorname{Cov}(W_t^i, W_t^j) = 
\begin{cases}
t &  i = j, \\
-\frac t2 &  |i-j| = 1, \\
0 & \text{otherwise},
\end{cases}
\]
$B_t^2$ is a standard one-dimensional Brownian motion independent of $\mathbf{B}^1_t$ and $\alpha := \max\{\frac 2k - \frac{2p}{k-1}, 0\}$.
\end{theorem}

\begin{theorem}[Main theorem, version 2]
\label{main_theorem2}
As $\varepsilon \to 0$, the random variables $\varepsilon \widetilde{A}_{t/\varepsilon^2}$ converge in law (with respect to the $L^\infty$ metric on compact intervals) to a $k$-dimensional Brownian motion \[\widetilde{\mathbf{B}}_t = (V^1_t, \dots, V^k_t)\] with covariance
\[
\operatorname{Cov}(V_t^i, V_t^j) = 
\begin{cases}
(\frac 1{k^2} - \frac{\alpha}{2k} + \frac {\alpha}{2})t &  i = j, \\
(\frac 1{k^2} - \frac{\alpha}{2k})t &  i\ne j, 
\end{cases}
\]
where $\alpha := \max\{\frac 2k - \frac{2p}{k-1}, 0\}$.
\end{theorem}

It can be verified that $(\mathbf{B}_{\alpha t}^1, B_t^2) = M \widetilde{\mathbf{B}}_t$ in distribution, so it is not hard to see that the two theorems are indeed equivalent.

Theorem~\ref{main_theorem} is a direct generalization of \cite[Theorem~2.5]{She11}. We will focus on proving this version in later sections. We noted that $\mathcal{C}_n$ is a simple random walk independent of $\mathcal{D}^{ij}_n$, so it scales to $B_t^2$ which is independent of $\mathbf{B}_t^1$ as in the theorem. Moreover, the value of $\alpha$ suggests that a phase transition happens at $p = 1- \frac 1k$, so that when $p$ gets larger than this value, the process $\widetilde{A}_n$ looks like a 1-dimensional brownian motion when viewed from a large scale). It will be further explained in the next section.

To see that the limit in Theorem~\ref{main_theorem2} is reasonable, we consider the special case $p = 0$, i.e., there are no ``flexible'' orders. In this case, $\widetilde{A}_n$ is a simple random walk on $\mathbb{Z}^k$, so we expect the limit to be a standard $k$-dimensional Brownian motion. Indeed, if $p = 0$, then $\alpha = 2/k$ and 
\[
\operatorname{Cov}(V_t^i, V_t^j) = 
\begin{cases}
\frac 1k &  i = j, \\
0 &  i\ne j.
\end{cases}
\]

\section{Computation of the covariance matrix and the critical value} \label{covariance}

In this section, we calculate the covariance matrix $[\operatorname{Cov}(\mathcal{D}^{i,i+1}_n, \mathcal{D}^{j,j+1}_n)]_{ij}$ where $1 \le i,j \le k-1$. It determines the value of $\alpha$, the critical value of $p$ and the covariance matrix of the limiting Brownian motion as in Theorem~\ref{main_theorem}.

\subsection{First calculations} \label{first_calculations}

Following the argument in \cite[Section~3.1]{She11}, we let $J$ be the smallest positive integer for which $X(-J,-1)$ contains exactly one burger (which is the rightmost burger in the semi-infinite stack $X(-\infty,-1)$). We use $|W|$ to denote the length of a reduced word $W$ and let $\chi = \chi(p) = \mathbb{E}[|X(-J,-1)|].$

The orders in $X(-J,-1)$ are of types different from the one burger in $X(-J,-1)$. In particular, we have that
\begin{equation} \label{2.1}
|\mathcal{D}^{ij}(-J,-1)| \leq |X(-J,-1)| = -\mathcal{C}(-J,-1) + 2.
\end{equation}

Since $\mathcal{C}(-n,-1)$ is a martingale in $n$, for a fixed $n$ the optional stopping theorem applied to the stopping time $J \land n$ implies that
\begin{equation} \label{2.2}
0 = \mathbb{E}[\mathcal{C}(-1,-1)] = \mathbb{E}[\mathcal{C}(-J,-1)\mathbbm{1}_{J \leq n}] + \mathbb{E}[\mathcal{C}(-n,-1) \mathbbm{1}_{J > n}].
\end{equation}
In the case $J>n$, $\mathcal{C}(-n,-1) \leq 0$, so $\mathbb{E}[\mathcal{C}(-J,-1)\mathbbm{1}_{J \leq n}] \geq 0$. Letting $n \to \infty$, we see that $\mathbb{E}[\mathcal{C}(-J,-1)] \geq 0$. On the other hand, $\mathbb{E}[\mathcal{C}(-J,-1)] \leq 1$, so by \eqref{2.1},
\begin{equation} \label{2.3}
\chi = \mathbb{E}[|X(-J,-1)|] \in [1,2].
\end{equation}
Note that $\chi = 2$ if and only if $\mathbb{E}[\mathcal{C}(-J,-1)] = 0$. Therefore, as $n \to \infty$ in \eqref{2.2}, we deduce that
\begin{equation} \label{2.4}
\chi = 2 \text{ if and only if } \lim_{n \to \infty} \mathbb{E}[\mathcal{C}(-n,-1) \mathbbm{1}_{J > n}] = 0.
\end{equation}

By \eqref{2.1}, \eqref{2.3} and symmetry, $\mathbb{E}[\mathcal{D}^{ij}(-J,-1)]$ exists and equals zero.
Moreover, since $|\mathcal{D}^{ij}(-n,-1)| \leq -\mathcal{C}(-n,-1)$ for $n < J$, by \eqref{2.4},
\begin{equation} \label{2.6}
\chi = 2 \text{ implies that } \lim_{n \to \infty} \mathbb{E}[|\mathcal{D}^{ij}(-n,-1)|\mathbbm{1}_{J>n}] = 0.
\end{equation}

It turns out that there is a dichotomy between $\chi = 2$ and $1 \leq \chi < 2$, which corresponds exactly to the phase transition at $p = 1 - 1/k$. In this section, we focus on the case $\chi = 2$ and show that $p\leq 1 - 1/k$. We leave the case $1 \leq \chi < 2$ to the following sections.

\subsection{Computation of \texorpdfstring{$\mathbb{E}[\mathcal{D}^{ij}(0) \mathcal{D}^{lm}(-J,-1)]$}{E[D\^{}ij(0) D\^{}lm(-J,-1)]}}

In preparation for computing the covariance $\operatorname{Cov}(\mathcal{D}^{ij}_n, \mathcal{D}^{lm}_n) = \mathbb{E}[\mathcal{D}^{ij}_n \mathcal{D}^{lm}_n]$ for any $i \neq j$ and $l \neq m$, we first calculate $\mathbb{E}[\mathcal{D}^{ij}(0) \mathcal{D}^{lm}(-J,-1)]$.

If $i, j, l$ and $m$ are distinct, then $\mathcal{D}^{ij}(0)$ is independent of $\mathcal{D}^{lm}(-J,-1)$, so by symmetry
\begin{equation} \label{2.2.1}
\mathbb{E}[\mathcal{D}^{ij}(0) \mathcal{D}^{lm}(-J,-1)] = 0.
\end{equation}

Next, we evaluate $\mathbb{E}[\mathcal{D}^{ij}(0) \mathcal{D}^{ij}(-J,-1)]$ for $i \neq j$.
On the event $X(0) \neq \squared{F}$, $\mathcal{D}^{ij}(0)$ is determined by $X(0)$ independently of $\mathcal{D}^{ij}(-J,-1)$, so $\mathbb{E}[\mathcal{D}^{ij}(0)\mathcal{D}^{ij}(-J,-1)] = 0$ by symmetry.

On the event $X(0) = \squared{F}$, we have $\phi(0)=-J$. Suppose $Y(0) = \squared{i}$. Then for any $j \neq i$, $\mathcal{D}^{ij}(0) = -1$, and for any other $j,l$, $\mathcal{D}^{jl}(0) = 0$. Because $X(-J,-1)$ contains a burger $i$ and (possibly) orders of types other than $i$, it follows that
\begin{align*}
|X(-J,-1)| + k - 2 =& \sum_{j \neq i} \mathcal{D}^{ij}(-J,-1) \\
=& - \sum_{j \neq i} \mathcal{D}^{ij}(0) \mathcal{D}^{ij}(-J,-1) 
= -\frac 12 \sum_{j \neq l} \mathcal{D}^{jl}(0) \mathcal{D}^{jl}(-J,-1).
\end{align*}
Taking the expectation of the above equation which does not depend on $i$, we see that conditioned on $X(0) = \squared{F}$,
\begin{equation} \label{2.7}
\chi + k - 2 = -\frac 12 \sum_{j \neq l} \mathbb{E}[\mathcal{D}^{jl}(0) \mathcal{D}^{jl}(-J,-1)] = -\frac{k(k-1)}2 \mathbb{E}[\mathcal{D}^{jl}(0) \mathcal{D}^{jl}(-J,-1)]
\end{equation}
by symmetry, where $j \neq l$ are arbitrary. Together with the case $X(0) \neq \squared{F}$, \eqref{2.7} implies that for any $i \neq j$,
\begin{equation} \label{2.8}
\mathbb{E}[\mathcal{D}^{ij}(0)\mathcal{D}^{ij}(-J,-1)] = - \frac{p(\chi+k-2)}{k(k-1)},
\end{equation}
since $X(0) = \squared{F}$ with probability $p/2$.

It remains to compute $\mathbb{E}[\mathcal{D}^{ij}(0) \mathcal{D}^{il}(-J,-1)]$ for distinct $i,j$ and $l$. On the event $X(0) \neq \squared{F}$, because of the independence of $\mathcal{D}^{ij}(0)$ and $\mathcal{D}^{il}(-J,-1)$, we have that  $\mathbb{E}[\mathcal{D}^{ij}(0) \mathcal{D}^{il}(-J,-1)] = 0$ as before. On the event $X(0) = \squared{F}$ and $Y(0) \neq \squared{i}$ or $\squared{j}$, we have $\mathcal{D}^{ij}(0) = 0$, so $\mathbb{E}[\mathcal{D}^{ij}(0) \mathcal{D}^{il}(-J,-1)] = 0$. On the event $X(0) = \squared{F}$ and $Y(0) = \squared{j}$, we have $\mathcal{D}^{ij}(0) = 1$, so $\mathbb{E}[\mathcal{D}^{ij}(0) \mathcal{D}^{il}(-J,-1)] = 0$. Finally, on the event $X(0) = \squared{F}$ and $Y(0) = \squared{i}$, we observe that $\mathcal{D}^{ij}(0) = \mathcal{D}^{il}(0) = -1$, so $\mathbb{E}[\mathcal{D}^{ij}(0) \mathcal{D}^{il}(-J,-1)] = \mathbb{E}[\mathcal{D}^{il}(0) \mathcal{D}^{il}(-J,-1)]$. Summarizing the cases above, we obtain that
\begin{equation} \label{2.2.2}
\mathbb{E}[\mathcal{D}^{ij}(0) \mathcal{D}^{il}(-J,-1)] = \mathbb{E}[\mathcal{D}^{il}(0) \mathcal{D}^{il}(-J,-1) \mathbbm{1}_{X(0) = \text{\scriptsize \squared{F}}, Y(0) = \text{\scriptsize \squared{i}}}].
\end{equation}

Since $\mathcal{D}^{il}(0) \mathcal{D}^{il}(-J,-1) = \mathcal{D}^{li}(0) \mathcal{D}^{li}(-J,-1)$ and $\mathcal{D}^{il}(0) = 0$ if $Y(0) \neq \squared{i}$ or $\squared{l}$,
\begin{align*}
\mathbb{E}[\mathcal{D}^{il}(0) \mathcal{D}^{il}(-J,-1) \mathbbm{1}_{X(0) = \text{\scriptsize \squared{F}}, Y(0) = \text{\scriptsize \squared{i}}}]
&= \mathbb{E}[\mathcal{D}^{il}(0) \mathcal{D}^{il}(-J,-1) \mathbbm{1}_{X(0) = \text{\scriptsize \squared{F}}, Y(0) = \text{\scriptsize \squared{l}}}] \\
&= \frac 12 \mathbb{E}[\mathcal{D}^{il}(0) \mathcal{D}^{il}(-J,-1) \mathbbm{1}_{X(0) = \text{\scriptsize \squared{F}}}] \\
&= \frac 12 \mathbb{E}[\mathcal{D}^{il}(0) \mathcal{D}^{il}(-J,-1)].
\end{align*}
Together with \eqref{2.2.2} and \eqref{2.8}, this implies that
\begin{equation} \label{2.2.3}
\mathbb{E}[\mathcal{D}^{ij}(0) \mathcal{D}^{il}(-J,-1)] = \frac 12 \mathbb{E}[\mathcal{D}^{ij}(0) \mathcal{D}^{ij}(-J,-1)] = - \frac{p(\chi+k-2)}{2k(k-1)}.
\end{equation}

\subsection{The covariance matrix and the phase transition} \label{covariance_matrix}

Conditional on the event $J < n$, $\mathcal{D}^{lm}(-n, -J-1)$ is independent of $\mathcal{D}^{ij}(0)$ because even if $X(0)$ were \squared{F} it would consume a burger after time $-J$. Therefore we have that $\mathbb{E}[\mathcal{D}^{ij}(0) \mathcal{D}^{lm}(-n, -J-1) \mathbbm{1}_{J < n}] = 0$ and it is not hard to see that
%thus $\mathbb{E}[\mathcal{D}^{ij}(0) \mathcal{D}^{lm}(-n, -1) \mathbbm{1}_{J \le n}] = \mathbb{E}[\mathcal{D}^{ij}(0) \mathcal{D}^{lm}(-J, -1) \mathbbm{1}_{J  \le n}]$.
\begin{equation} \label{2.3.1}
\mathbb{E}[\mathcal{D}^{ij}(0) \mathcal{D}^{lm}(-n,-1)] = \mathbb{E}[\mathcal{D}^{ij}(0) \mathcal{D}^{lm}(-J,-1) \mathbbm{1}_{J \leq n}] + \mathbb{E}[\mathcal{D}^{ij}(0) \mathcal{D}^{lm}(-n,-1) \mathbbm{1}_{J > n}]
\end{equation}
where the rightmost term tends to zero as $n \to \infty$ if $\chi = 2$ because of \eqref{2.6}. Therefore, summarizing \eqref{2.2.1}, \eqref{2.8} and \eqref{2.2.3}, we see that for $i \neq j$, $l \neq m$,
\begin{equation} \label{2.9}
\chi = 2 \text{ implies } \lim_{n \to \infty} \mathbb{E}[\mathcal{D}^{ij}(0) \mathcal{D}^{lm}(-n,-1)] = 
\begin{cases}
- \frac p{k-1} &  i = l, j = m, \\
- \frac p{2(k-1)} &  i = l, j \neq m, \\
0 &  i, j, l, m \text{ distinct}.
\end{cases}
\end{equation}

Moreover, $\mathcal{D}^{ij}(0)^2 = 1$ if $Y(0)$ is of type $i$ or $j$, and $\mathcal{D}^{ij}(0)\mathcal{D}^{il}(0) = 1$ if $Y(0)$ is of type $i$, so
\begin{equation} \label{2.2.4}
\mathbb{E}[\mathcal{D}^{ij}(0)\mathcal{D}^{lm}(0)] =
\begin{cases}
\frac 2k &  i = l, j = m, \\
\frac 1k &  i = l, j \neq m, \\
0 &  i, j, l, m \text{ distinct}.
\end{cases}
\end{equation}

Now we evaluate $\operatorname{Cov}(\mathcal{D}^{ij}_n,\mathcal{D}^{lm}_n) = \mathbb{E}[\mathcal{D}^{ij}_n \mathcal{D}^{lm}_n]$. Using 
\[\mathcal{D}^{ij}_r \mathcal{D}^{lm}_r = \mathcal{D}^{ij}(r)\mathcal{D}^{lm}(r) + \mathcal{D}^{ij}(r)\mathcal{D}^{lm}_{r-1} + \mathcal{D}^{ij}_{r-1}\mathcal{D}^{lm}(r) + \mathcal{D}^{ij}_{r-1}\mathcal{D}^{lm}_{r-1} \]
recursively for $2\leq r \leq n$ and applying the translation invariance of the law of $Y_m$, we deduce that when $\chi = 2$,
\begin{align} 
&\operatorname{Cov}(\mathcal{D}^{ij}_n,\mathcal{D}^{lm}_n) \nonumber\\
=& \sum_{r=1}^n \mathbb{E}[\mathcal{D}^{ij}(r)\mathcal{D}^{lm}(r)] + \sum_{r=2}^n \mathbb{E}[\mathcal{D}^{ij}(r)\mathcal{D}^{lm}_{r-1} + \mathcal{D}^{ij}_{r-1}\mathcal{D}^{lm}(r)] \nonumber \\
=& n \mathbb{E}[\mathcal{D}^{ij}(0)\mathcal{D}^{lm}(0)] + \sum_{r=2}^n \big(\mathbb{E}[\mathcal{D}^{ij}(0)\mathcal{D}^{lm}(1-r,-1)] + \mathbb{E}[\mathcal{D}^{lm}(0)\mathcal{D}^{ij}(1-r,-1)]\big) \nonumber \\
=&
\begin{cases}
\frac{2n}k - \frac {2np}{k-1} + o(n) &  i = l, j = m, \\
\frac nk - \frac {np}{k-1} + o(n) &  i = l, j \neq m, \\
o(n) &  i, j, l, m \text{ distinct},
\end{cases}
\label{2.10} 
\end{align}
where the last equation follows from \eqref{2.9} and \eqref{2.2.4}.

For $i = l$ and $j = m$, the variance is nonnegative, so
\begin{equation} \label{2.11}
\chi = 2 \text{ implies } p \leq 1-\frac 1k.
\end{equation}

We remark that \eqref{2.10} and \eqref{2.11} suggest that the phase transition happens at the critical value $p = 1 - \frac 1k$. Let $\alpha = \max\{\frac 2k - \frac{2p}{k-1},0\}$. When $\chi = 2$ and $p \leq 1- \frac 1k$, it follows immediately from \eqref{2.10} that
\begin{equation}
\operatorname{Cov}(\mathcal{D}^{i,i+1}_n,\mathcal{D}^{j,j+1}_n) =
\begin{cases}
\alpha n + o(n) &  i = j, \\
-\frac {\alpha n}2 + o(n) &  |i-j|=1, \\
o(n) & \text{otherwise}.
\end{cases}
\label{2.2.5} 
\end{equation}
This explains why the limiting Brownian motion should have the covariance matrix as in Theorem~\ref{main_theorem}. In the following sections, we will take care of the case $\chi < 2$ and prove that the convergence indeed happens.

\section{Excursion words revisited} \label{excursion}

This section generalizes the discussion of excursion words in \cite[Section~3.3]{She11} to the $k$-burger case. The proof structure and most arguments are largely based on those in the original paper. Since adaptation is required throughout the proof, we include most details for completeness.

First, we quote two results \cite[Lemma~3.3 and 3.4]{She11} directly:

\begin{lemma} 
\label{lemma3.1}
Let $Z_1, Z_2, Z_3, \dots$ be $i.i.d.$ random variables on some measure space and $\psi$ a measurable function on that space such that $\mathbb{E}[\psi(Z_1)]<\infty$. Let $T$ be stopping time of the process $Z_1, Z_2,\dots$ and $\mathbb{E}[T]<\infty$. Then $\mathbb{E}[\sum_{j=1}^T \psi(Z_j)]<\infty$.
\end{lemma}

\begin{lemma}
\label{lemma3.2}
Let $Z_1, Z_2,\dots$ be $i.i.d.$ random variables on some measure space and let $\mathcal{Z}_n$ be a non-negative integer-valued process adapted to the filtration of the $Z_n$ (i.e., each $\mathcal{Z}_n$ is a function of $Z_1, Z_2, \dots, Z_n$) that has the following properties:

\begin{enumerate}
\item \emph{Bounded initial expectation:} $\mathbb{E}[\mathcal{Z}_1]<\infty$.

\item \emph{Positive chance to hit zero when close to zero:} For each $k>0$ there exists a positive chance $p_k$ such that conditioned on any choice of $Z_1, Z_2,\dots , Z_n$ for which $\mathcal{Z}_n=k$, the conditional probability that $\mathcal{Z}_{n+1}=0$ is at least $p_k$.

\item \emph{Uniformly negative drift when far from zero:} There exist positive constants $C$ and $c$ such that if we condition on any choice of $Z_1, Z_2, \dots, Z_n$ for which $\mathcal{Z}_n\geq C$, the conditional expectation of $\mathcal{Z}_{n+1}-\mathcal{Z}_n$ is less than $-c$. 

\item \emph{Bounded expectation when near zero:} There further exists a constant $b$ such that if we condition on any choice of $Z_1,Z_2,\dots,Z_n$ for which $\mathcal{Z}_n<C$, then the conditional expectation of $\mathcal{Z}_{n+1}$ is less than $b$.
\end{enumerate}

Then $\mathbb{E}[\operatorname{min} \{n:\mathcal{Z}_n=0\}]<\infty$.
\end{lemma}

Let $E = X(1,K)$ where $K$ is the smallest integer such that $\mathcal{C}_{K+1} < 0$ and call $E$ an \emph{excursion word}.
If $i$ is positive, let $V_i$ be the symbol corresponding to the $i$th record minimum of $\mathcal{C}_n$, counting forward from zero. If $i$ is negative, let $V_i$ be the $-i$th record minimum of $\mathcal{C}_n$, counting backward from zero. Denote by $E_i$ the reduced word in between $V_{i-1}$ and $V_i$ (or in between $0$ and $V_i$ if $i=1$). Note that $E = E_1$.

It is easy to check that $E$ almost surely contains no \squared{F} symbols and there are always as many burgers as orders in the word $E$. In addition, $E_i$'s and $E$ are i.i.d. excursion words. The following lemma also holds:

\begin{lemma} \label{lemma3.4}
If $p$ is such that $\chi<2$, then the expected word length $\mathbb{E}[|E|]$ is finite, and hence the expected number of symbols in $E$ of each type in \{\emph{\circled{1}, \dots , \circled{k}, \squared{1}, \dots , \squared{k}}\} is $\mathbb{E}[|E|]/(2k).$
\end{lemma}

Since $E$ is balanced between burgers and orders, the second statement follows from the first immediately by symmetry. For the first statement, it suffices to prove that the expected number of burgers in $E_{-1}$ is finite, since $E$ and $E_{-1}$ have the same distribution. The original proof still works, so we omit it.

Next, we consider the following sequences: 

\begin{enumerate}

\item \emph{$m$-th empty order stack}: let $O_m$ be the $m$-th smallest value of $j \geq 0$ with the property that $X(-j,0)$ has an empty order stack. 

\item \emph{$m$-th empty burger stack}: $B_m$ is the $m$-th smallest value of $j \geq 1$ with the property that $X(1,j)$ has an empty burger stack.

\item \emph{$m$-th left record minimum}: $L_m = L_m^0$ is the smallest value of $j \geq 0$ such that $\mathcal{C}(-j,0)=m$. Thus, $X(-L_m,0)=\overline{V_{-m}E_{-m}\dots V_{-1}E_{-1}}$.

\item \emph{$m$-th right record minimum}: $R_m = R_m^0$ is the smallest value $j \geq 1$ such that $\mathcal{C}(1,j) = -m$. Thus,
$X(1,R_m) = \overline{E_1V_1\dots E_mV_m}.$

\item \emph{$m$-th left minimum with no orders of type $1, 2,\dots, i$}: for $1 \leq i \leq k$, $L^i_m$ is the $m$-th smallest value of $j\geq 0$ with the property that $j = L_{m'}$ for some $m'$ and $X(-j,0)$ has no orders of type $1, 2,\dots, i$.

\item \emph{$m$-th right minimum with no burgers of type $1, 2,\dots, i$}: for $1 \leq i \leq k$, $R^i_m$ is the $m$-th smallest value of $j\geq 1$ with the property that $j = R_{m'}$ for some $m'$ and $X(1,j)$ has no burgers of type $1, 2,\dots, i$.

\end{enumerate}

We observe that all these record sequences have the property that the words between two consecutive records are i.i.d.. Moreover, for $1 \leq i \leq k$, each $L^i_m$ is equal to $L^{i-1}_{m'}$ for some $m'$ by definition. Thus we can write each $X(-L^i_m,-L^i_{m-1}-1)$ as a product of consecutive words of the form $X(-L^{i-1}_{m'}, -L^{i-1}_{m'-1}-1)$. We have the following lemma:

\begin{lemma} \label{lemma3.5}
The following are equivalent:

\begin{enumerate}

\item $\mathbb{E}[|E|]<\infty$;

\item $\mathbb{E}[|X(-L^i_1,0)|]<\infty$ where $0 \leq i \leq k$;

\item $\mathbb{E}[|X(-O_1,0)|]<\infty$;

\item $\mathbb{E}[|X(1,R_1^i)|]<\infty$ where $0 \leq i \leq k$;

\item $\mathbb{E}[|X(1,B_1)|]<\infty$.

\end{enumerate}
\end{lemma}

\begin{proof} 
\emph{1 implies 2:} Note that for $i = 0$, $L_1^0 = L_1$ and $X(-L_1^0,0) = \overline{V_{-1}E_{-1}}$. Since $E_{-1}$ and $E$ have the same law, 2 follows immediate from 1 when $i$ = 0. To prove 2 for $1 \leq i \leq k$, we use induction.

Assume 2 holds for $i-1$. Let $H(m)$ be the number of orders of type $i$ in $X(-L_m^{i-1},0)$. If we can apply Lemma~\ref{lemma3.2} with $Z_m = X(-L_m^{i-1},-L_{m-1}^{i-1}-1)$ and $\mathcal{Z}_m = H(m)$, then $\mathbb{E}[\min\{m:H(m)=0\}] < \infty$. That means the expected number of $X(-L_m^{i-1},-L_{m-1}^{i-1}-1)$ concatenated to produce $X(-L_1^i,0)$ is finite. Since $X(-L_m^{i-1},-L_{m-1}^{i-1}-1)$ are identically distributed as $X(-L_1^{i-1},0)$ which has finite expected length by inductive hypothesis, Lemma~\ref{lemma3.1} implies that $X(-L_1^i,0)$ also has finite expected length.

Therefore it remains to check the four assumptions of Lemma~\ref{lemma3.2}. It is easy to see that Assumption 1, 2 and 4 follow from the construction of the sequence and the inductive hypothesis, so we focus on the negative drift assumption. For any $m>1$, \[H(m)=\operatorname{max} \{H(m-1)-h_m, 0\}+o_m,\] where $h_m$ is the number of burger $i$ in $X(-L_m^{i-1},-L_{m-1}^{i-1}-1)$ and $o_m$ is the number of order $i$ in it. The expected number of burger $i$ equals the expected number of order $i$ in $E_{-m}$ by Lemma~\ref{lemma3.4}, while the expected number of burger $i$ in $V_{-m}$ is $1/k$, which has no orders. Hence $\mathbb{E}[h_m] \geq \mathbb{E}[o_m] + 1/k$ since $X(-L_m^{i-1},-L_{m-1}^{i-1}-1)$ is a concatenation of at least one $\overline{V_{-m'}E_{-m'}}$. Note that
\[ H(m) - H(m-1) = o_m - h_m + (h_m - H(m-1)) \mathbbm{1}_{\{H(m-1)-h_m < 0\}} \]
and $\mathbb{E}[(h_m - j) \mathbbm{1}_{h_m > j}] \le \mathbb{E}[h_m \mathbbm{1}_{h_m > j}] \to 0$ as $j \to \infty$ by assumption. Thus there is $C>0$ such that $\mathbb{E} [ H(m) - H(m-1) | H(m-1) = j] \le -1/(2k)$ for $j > C$, so the negative drift assumption is verified.

\emph{2 implies 3:} By definition, $X(-O_1,0)$ corresponds to the first time that the stack contains only burgers, while $X(-L_1^k,0)$ corresponds to the first time that the stack contains only burgers and increases in length, it follows easily that $|X(-O_1,0)| \leq |X(-L_1^k,0)|$, so the expectation is finite.

\emph{3 implies 1:} The number of burgers in $X(-O_1,0)$ is at least the number of burgers in $E_{-1}$, which accounts for half of its length, so $\mathbb{E}[|E_{-1}|] < \infty$. Thus the same holds for $E$.

The equivalence of 1, 4 and 5 are proved similarly. 
\end{proof}

The next lemma on the asymptotic fractions of burgers and orders is key to the proof of the main theorem.

\begin{lemma} \label{lemma3.6}

If $\mathbb{E}[|E|]<\infty$, then as $n\rightarrow \infty$ the fraction of \emph{\circled{i}} symbols among the rightmost $n$ elements of $X(-\infty,0)$ tends to $1/k$ almost surely for any $i$. Also, as $n\rightarrow \infty$ the fraction of \emph{\squared{i} or \squared{F}} symbols among the leftmost $n$ elements of $X(1,\infty)$ tends to some positive constant almost surely.

On the other hand if $\mathbb{E}[|E|]=\infty$, then as $n\rightarrow \infty$ the fraction of \emph{\squared{F}} symbols among the leftmost $n$ elements of $X(1,\infty)$ tends to zero almost surely.

\end{lemma}

\begin{proof}
%The proof is almost the same as in \cite{She11}. 
If $\mathbb{E}[|E|]<\infty$, then the words $X(-O_m, -O_{m-1}-1)$ are i.i.d. with finite expectations by Lemma~\ref{lemma3.5}. Hence $X(-\infty,0)$ is a concatenation of i.i.d. words $X(-O_m, -O_{m-1}-1)$. The law of large numbers implies that the number of each type of burgers in $X(-O_m, 0)$ is given by $Cm+o(m)$ almost surely for some constant $C$. By symmetry, these constants are all equal to $\mathbb{E}[|X(-O_1,0)|]/k$. The first statement then follows, and the second is proved analogously.

For the last statement, we note that $X(1,\infty)$ is an i.i.d. concatenation of burger-free words $X(B_{m-1}+1,B_m)$, and an \squared{F} symbol can be added only when the burger stack is empty. Hence the number of \squared{F} symbols in $X(1, B_m)$ grows like a constant times $m$. If $\mathbb{E}[|E|]=\infty$, Lemma~\ref{lemma3.5} implies that $\mathbb{E}[|X(1,B_1)|]=\infty$. Thus the number of orders in $X(1, B_m)$ grows faster than any constant multiple of $m$ almost surely, so the fraction of \squared{F} symbols tends to zero almost surely. 
\end{proof}

\section{Bounded increments and tail estimates} \label{tail_estimates}

We fix a semi-infinite stack $S_0 = X(-\infty,0)$ and let $X(1), X(2), \dots$ be chosen according to $\mu$. An analogy of \cite[Lemma 3.10]{She11} still holds in this case, but it requires a different proof as we will see.

\begin{lemma} \label{lemma4.1}
For $N>0$, $\mathbb{E}[\mathcal{D}^{ij}_N|X(l):1 \le l \le n]$ and $\mathbb{E}[\mathcal{D}^{ij}_N|X(l):1 \le l \le n, \mathcal{C}_l:l \leq N]$ are both martingales in $n$ with increments of magnitude at most two.
\end{lemma}

Instead of monotonicity properties of stacks used in \cite{She11} which do not generalize to higher dimensions, we introduce the notion of \emph{neighbor stacks} which allows us to prove a similar result. 

\begin{definition}
Two semi-infinite stacks $S_0$ and $S_1$ are called neighbors if $S_1$ can be achieved from $S_0$ by removing an arbitrary burger from $S_0$, or vice versa.
\end{definition}

For example, $S_0 = \cdots \circled{2}\, \circled{1}\, \circled{1}\, \circled{3}\, \circled{2}\, \circled{2}\, \circled{3}$ and $S_1 = \cdots \circled{2}\, \circled{1}\, \circled{1}\, \circled{2}\, \circled{2}\, \circled{3}$ are neighbors, because one can get $S_1$ from $S_0$ by removing the fourth burger from the right.

\begin{lemma} \label{lemma4.3}
If $S_0$ and $S_1$ are neighbors, then for any word $W$, $\overline{S_0 W}$ and $\overline{S_1 W}$ are still neighbors.
\end{lemma}

\begin{proof}
Assume that we get $S_1$ from $S_0$ by deleting a $\circled{j}$. By induction, we may also assume that $W$ contains a single element.

If $W$ is a burger, then for $\sigma = 1,2$, $\overline{S_\sigma W}$ is achieved by adding $W$ onto $S_\sigma$. If $W = \squared{F}$, then $\overline{S_\sigma W}$ is achieved by deleting the rightmost burger from $S_\sigma$. If $W = \squared{i}$, then $\overline{S_\sigma W}$ is achieved by deleting the rightmost \circled{i} from $S_\sigma$. Hence in these three cases, it is easily seen that the resulting two stacks are still neighbors.

If $W = \squared{j}$ and there is a \circled{j} in $S_0$ to the right of the \circled{j} which we deleted to get $S_1$, then $\overline{S_\sigma W}$ is achieved by deleting the rightmost \circled{j} from $S_\sigma$. Hence the resulting two stacks are neighbors. Otherwise, the \circled{j} deleted to get $S_1$ is the rightmost \circled{j} in $S_0$, so $\overline{S_0 W} = S_1$. Hence $\overline{S_0 W}$ and $\overline{S_1 W}$ are neighbors.
\end{proof}

\begin{proof}[Proof of Lemma~\ref{lemma4.1}]
Since the two conditional expectations are clearly martingales in $n$, we only need to prove that the increments are bounded. To this end, it suffices to show that changing $X(l)$ for a single $1 \leq l \leq N$ only changes $\mathcal{D}^{ij}_N$ by at most two. 

Suppose that $X(l)$ is changed to $X(l)'$. Here we make the convention that a product of words is always reduced. It is easy to see that $X(-\infty,l)$ and $X(-\infty,l-1)X(l)'$ have a common neighbor $X(-\infty,l-1)$. Lemma~\ref{lemma4.3} then implies that $X(-\infty,N)$ and $X(-\infty,l-1)X(l)'X(l+1,N)$ have a common neighbor $X(-\infty,l-1)X(l+1,N)$. Since the $ij$-discrepancy differs by at most one between neighbors, we see that $\mathcal{D}^{ij}_N$ changes by at most two if we change a single $X(l)$.
\end{proof}

The following tail estimates are adapted from \cite[Lemma~3.12 and 3.13]{She11}.

\begin{lemma} \label{lemma4.4}
Fix any $p \in [0,1]$ and a semi-infinite stack $S_0 = X(-\infty,0)$. There exist positive constants $C_1$ and $C_2$ such that for any choice of $S_0$, $a>0$, $n>1$ and any $i,j$,
\[ \mathbb{P}(\max_{1 \leq l \leq n} |\mathcal{C}_l| > a\sqrt{n}) \le C_1 e^{-C_2 a}
\quad \text{ and } \quad  \mathbb{P}(\max_{1 \leq l \leq n} |\mathcal{D}^{ij}_l| > a\sqrt{n}) \le C_1 e^{-C_2 a}. \]
\end{lemma}

The original proof carries over almost verbatim.
The idea is that Lemma~\ref{lemma4.1} gives bounded increments of the martingales, so we can apply a pre-established tail estimate of martingales with bounded increments.
We remark that it is an important technique to estimate the tails of martingales with bounded jumps. See \cite{Dem96} for more interesting results.

\begin{lemma} \label{lemma4.5}
Fix any $p \in [0,1]$. There exist positive constants $C_1$ and $C_2$ such that for any $a>0$ and $n>1$,
\[ \mathbb{P}(|X(1,n)| > a\sqrt{n}) \le C_1 e^{-C_2 a}. \]
\end{lemma}

\begin{proof}
Let the semi-infinite stack $S_0$ be rotating among $\circled{1}, \dots, \circled{k}$. Suppose that $\mathcal{C}_l$ and all $\mathcal{D}^{ij}_l$ fluctuate by at most $a\sqrt{n}/(4k-1)$ for $1 \leq l \leq n$.

Claim that no burger in $S_0$ expect the rightmost $a\sqrt{n}(2k-1)/(4k-1)$ burgers will be consumed in the first $n$ steps. Assume the opposite. If the first such burger is consumed at step $l$ and is an \circled{m}, then at this moment all burgers to the right are of types different from \circled{m}. Since $\mathcal{C}_l \geq -a\sqrt{n}/(4k-1)$, there are at least $a\sqrt{n}(2k-2)/(4k-1)$ burgers above the \circled{m}. Among them there are at least $2a\sqrt{n}/(4k-1)$ burgers of some type $m' \ne m$. Hence $|\mathcal{D}^{mm'}_l| > a\sqrt{n}/(4k-1)$, which is a contradiction.

It follows from the claim that there are at most $a\sqrt{n}(2k-1)/(4k-1)$ orders in $X(1,n)$. Since $\mathcal{C}_l$ fluctuates by at most $a\sqrt{n}/(4k-1)$, there are at most $2ka\sqrt{n}/(4k-1)$ burgers in $X(1,n)$. Therefore, $|X(1,n)| \leq a\sqrt{n}$. 

Thus, to have $|X(1,n)| > a\sqrt{n}$, $\mathcal{C}_l$ or at least one $\mathcal{D}^{ij}_l$ must fluctuate by more than $a\sqrt{n}/(4k-1)$. An application of Lemma~\ref{lemma4.4} then completes the proof.
\end{proof}

\section{Proof of the main theorem} \label{main_proof}

The proof parallels that in \cite[Section~3.5 and 3.6]{She11}.

\subsection{The case \texorpdfstring{$\chi <2$}{chi<2}} \label{chi<2}

In this subsection, we will resolve the remaining case from Section~\ref{covariance}, i.e., the case $\chi<2$. We will use the results from Section~\ref{excursion} and \ref{tail_estimates} to prove that when $\chi<2$, the scaling limit of $A_n$ on a compact interval has the law of a one-dimensional Brownian motion. This means that the total burger count $\mathcal{C}_n$ dominates. As we remarked after the statement of Theorem~\ref{main_theorem}, $\mathcal{C}_n$ is a simple random walk and thus scales to a Brownian motion, so it suffices to show that $\mathcal{D}^{ij}_n$ scales to $0$ in law on compact intervals.

In addition to the statement above, we will show that $\chi<2$ implies that $p>1-1/k$. Together with \eqref{2.11}, this gives the dichotomy mentioned in Section~\ref{first_calculations}, namely,
\begin{equation} \label{dichotomy} 
\chi<2 \iff p>1-1/k \quad \text{ and } \quad \chi = 2 \iff p \leq 1 - 1/k.
\end{equation}
Thus this subsection proves Theorem~\ref{main_theorem} in the case $p > 1-1/k.$
We divide the proof into three lemmas.

\begin{lemma} \label{lemma5.1}
If $\mathbb{E}[|E|]<\infty$ (which holds when $\chi<2$), then $\operatorname{Var}[\mathcal{D}^{ij}_n]=o(n)$ for all pairs $(i,j)$.	
\end{lemma}

\begin{proof}
First, we prove that the random variables $n^{-1/2} \mathcal{D}^{ij}_n$ converge to 0 in probability. 
To do this, we consider the following events:
\begin{enumerate}
\item $|X(1,n)|<a\sqrt{n}$;
	
\item The top $2ka\sqrt{n}$ burgers in stack $X(-\infty,0)$ are well balanced among all burger types with error $\varepsilon \sqrt{n}$, i.e., the number of burgers of any type is between $(2a-\varepsilon)\sqrt{n}$ and $(2a+\varepsilon)\sqrt{n}$;

\item The top $b$ burgers in the stack $X(-\infty,n)$ are well balanced among all burger types with error $\varepsilon\sqrt{n}$ for all $b>(2k-1)a\sqrt{n}$.
\end{enumerate}
	
We assert that if all three events happen, then $|n^{-1/2}\mathcal{D}^{ij}_n|<4\varepsilon$. First, 1 and 2 together imply that all the orders in $X(1,n)$ are fulfilled by the top $2ka\sqrt{n}$ burgers in $X(-\infty,0)$, so the burgers below height $-2ka\sqrt{n}$ in $X(-\infty,0)$ are not affected by $X(1,n)$. Hence the stacks $X(-\infty,0)$ and $X(-\infty,n)$ are identical below height $-2ka\sqrt{n}$. On the other hand, $|X(1,n)|<a\sqrt{n}$ implies that $|\mathcal{C}_n|<a\sqrt{n}$, so the number of burgers in $X(-\infty,n)$ above height $-2ka\sqrt{n}$ is at least $(2k-1)a\sqrt{n}$. By 2 and 3, the discrepancies between two burger types above height $-2ka\sqrt{n}$ are less than $2\varepsilon\sqrt{n}$ for both stacks, so $|\mathcal{D}^{ij}_n|$ is at most $4\varepsilon\sqrt{n}$, as desired.
	
Next, we observe that all three events happen with high probability if we choose $a$ and $n$ properly. For fixed $\varepsilon>0$, we first choose $a$ large enough so that 1 happens with high probability using Lemma~\ref{lemma4.4}. Then by Lemma~\ref{lemma3.6}, we choose $n$ large enough so that 2 and 3 happen with high probability. 
	
Thus we conclude that $\lim_{n\rightarrow\infty}\mathbb{P}[|n^{-1/2}\mathcal{D}^{ij}_n|>\varepsilon]=0$ for all $\varepsilon>0$, i.e., $n^{-1/2}\mathcal{D}^{ij}_n$ converge to 0 in probability.
	
It remains to check that $\operatorname{Var}[n^{-1/2}\mathcal{D}^{ij}_n]=\mathbb{E}[n^{-1}(\mathcal{D}^{ij}_n)^2]$ tends to 0 as $n\rightarrow\infty$. This follows from the fact that $n^{-1}(\mathcal{D}^{ij}_n)^2$ tends to 0 in probability together with the uniform bounds on the tails given by Lemma~\ref{lemma4.4}. 
\end{proof}

The following two lemmas are proved in exactly the same way as \cite[Lemma~3.15 and 3.16]{She11}, so we omit the proofs.

\begin{lemma} \label{lemma5.2}
If $\operatorname{Var}[\mathcal{D}^{ij}_n]=o(n)$, then $n^{-1/2}\operatorname{max}\{|\mathcal{D}^{ij}_l|:1\leq l \leq nt \}$ converges to zero in probability as $n \to \infty$ for any fixed $t>0$.
\end{lemma}

The trick of the proof is to first divide the time interval into small subintervals, then observe the convergence at the end points, and finally use approximation to complete the proof.
Note that by Lemma~\ref{lemma5.2}, we immediately obtain that $A_n$ converges in law to a one-dimensional Brownian motion on compact intervals.

\begin{lemma} \label{lemma5.3}
If $\chi <2$ and $\operatorname{Var}[\mathcal{D}^{ij}_n]=o(n)$, then \[\lim_{n\rightarrow \infty}\mathbb{E}[|\mathcal{D}^{ij}(-n,-1)|\mathbbm{1}_{J>n}]=0.\]
\end{lemma}

Interested readers may refer to the proof in the original paper which involves introducing new measures via Radon-Nikodym derivatives and recentering the sequence. The original proof also uses the fact that one-dimensional random walk conditioned to stay positive scales to a three-dimensional Bessel process, which is explained by \cite{Pit75}.

Letting $n \to \infty$ in \eqref{2.3.1} and using Lemma~\ref{lemma5.3} and \eqref{2.8}, we deduce that \[\lim_{n\rightarrow \infty} \mathbb{E}[\mathcal{D}^{ij}(0)\mathcal{D}^{lm}(-n,-1)] = \mathbb{E}[\mathcal{D}^{ij}(0)\mathcal{D}^{lm}(-J,-1)] = - \frac{p(\chi+k-2)}{k(k-1)}. \]
Following the same computation as in \eqref{2.10}, we obtain that
\[ \operatorname{Var}(\mathcal{D}^{ij}_n) = \frac{2n}{k}-\frac{2np(\chi+k-2)}{k(k-1)}+o(n).\]By Lemma~\ref{lemma5.1}, we must have $\frac{2n}{k}=\frac{2np(\chi+k-2)}{k(k-1)}$, i.e., $p = \frac{k-1}{\chi+k-2}$. Hence $\chi<2$ implies that $p>1-1/k$, which gives us the promised dichotomy \eqref{dichotomy}.

\subsection{The case \texorpdfstring{$\chi=2$}{chi=2}} \label{chi=2}

It finally remains to prove the main theorem in the case $\chi=2$. First, if $p=1-1/k$, then $\operatorname{Var}[\mathcal{D}^{ij}_n]=o(n)$ by \eqref{2.10}, so the convergence follows from our argument in Section~\ref{chi<2}. 

Next, we may assume $p<1-1/k$, so that $\operatorname{Var}[\mathcal{D}^{ij}_n]\neq o(n)$. By the contrapositive of Lemma~\ref{lemma5.1}, we must have $\mathbb{E}[|E|]=\infty$.
Then we can apply the second part of Lemma~\ref{lemma3.6}, which asserts that the number of $\squared{F}$ symbols in $X(1,n)$ is small relative to the total number of orders in $X(1,n)$ as $n$ gets large. To be more precise, the number of \squared{F} in $X(1,\lfloor tn \rfloor)$ is $o(\sqrt{n})$ with probability tending to one as $n \to \infty$ by Lemma~\ref{lemma3.6} and Lemma~\ref{lemma4.5}. 
Therefore, for $t_1 + t_2 = t_3$, the laws of $A_{\lfloor t_1n\rfloor}$ and $A_{\lfloor t_2n\rfloor}$ add to the law of $A_{\lfloor (t_1+t_2)n\rfloor}$ up to an error of $o(\sqrt{n})$ with high probability. 

On the other hand, since the variances of the random variables $n^{-1/2}A_{tn}$ converge to constants as $n\rightarrow\infty$ for fixed $t$, at least subsequentially the random variables $n^{-1/2}A_{tn}$ converge in law to a limit. Moreover, if we choose a finite collection of $t$ values, namely $0<t_1<t_2<\dots<t_m<\infty$, the joint law of 
\[ \big(n^{-1/2}A_{\lfloor t_1n \rfloor},n^{-1/2}A_{\lfloor t_2n \rfloor},\dots, n^{-1/2}A_{\lfloor t_mn \rfloor} \big) \] 
also converges subsequentially to a limiting law. 

Now we combine the two observations above. We have that the law of $n^{-1/2}A_{\lfloor tn \rfloor}$ is equal to the law of the sum of $l$ independent copies of $n^{-1/2}A_{\lfloor tn/l \rfloor}$ plus a term which is $o(1)$ with high probability (since we have multiplied by $n^{-1/2}$). Hence, the subsequential weak limit of $n^{-1/2} A_{\lfloor tn \rfloor}$ must equal the sum of $l$ i.i.d. random variables. In particular, since $l$ is arbitrary, the limiting law has to be infinitely divisible. Note that the process $n^{-1/2}A_{\lfloor tn \rfloor}$ is almost surely continuous in $t$, so we conclude that the subsequential limit discussed above has to be a Gaussian with mean zero. We refer to \cite{Ber96} for more background on infinitely divisible processes, L{\'e}vy processes and Gaussian processes.

The covariance matrix of $n^{-1/2}A_{n}$ is already given by our calculation in Section~\ref{covariance}, and Lemma~\ref{lemma4.4} guarantees that $n^{-1/2}A_{\lfloor tn \rfloor}$ are tight, so the subsequential limit has the correct covariance matrix. We conclude that the limit indeed has the Gaussian distribution given in Theorem~\ref{main_theorem}. Moreover, our argument implies that any subsequence of $n^{-1/2}A_{tn}$ has a further subsequence converging in law to this Gaussian distribution, so the whole sequence converges to this law.

The same is true if we choose a finite collection of $t_i$'s, so the finite-dimensional joint law of 
\[ \big(n^{-1/2}A_{\lfloor t_1n \rfloor},n^{-1/2}A_{\lfloor t_2n \rfloor},\dots, n^{-1/2}A_{\lfloor t_mn \rfloor} \big) \] 
converges to a limiting law, which is exactly the law of $(\mathbf{W}_{t_1},\mathbf{W}_{t_2},\dots, \mathbf{W}_{t_m})$, where $\mathbf{W}_t$ is the $k$-dimensional Brownian motion $(\mathbf{B}^1_{\alpha t},B^2_t)$ described in Theorem~\ref{main_theorem}. 

The transition from a discrete collection of $t_i$'s to a compact interval follows similarly as in the proof of Lemma~\ref{lemma5.2}. As the maximum gap between $t_i$'s gets smaller, the probability that (the norm of) the fluctuation in some interval $[t_i,t_{i+1}]$ exceeds $\varepsilon$ tends to zero as $n\rightarrow\infty$ for both $n^{-1/2}A_{\lfloor tn \rfloor}$ and $\mathbf{W}_t$ where $t\in [0,t_m]$. Hence the two processes are uniformly close on the interval $[0,t_m]$ with probability tending to one as $n\rightarrow \infty$. Therefore, Theorem~\ref{main_theorem} is fully proved.

\bibliographystyle{alpha}
\bibliography{k-burger}

\end{document}